\documentclass[a4paper,10pt]{article}
\usepackage[latin1]{inputenc}
\usepackage[T1]{fontenc}
\usepackage{amsmath}
\usepackage{amsfonts}
\usepackage{amstext}
\usepackage{amsthm}
\usepackage[all]{xy}
\usepackage{geometry}
\usepackage{srcltx}
\usepackage{graphics}
\usepackage{epsfig}
\usepackage{subfigure}
\usepackage{slashed}
\usepackage{color}
\usepackage{amssymb}
\usepackage{srcltx}
\newtheorem{theorem}{Theorem}[section]
\newtheorem{lemma}[theorem]{Lemma}
\newtheorem{defi}[theorem]{Definition}

\newtheorem{prop}[theorem]{Proposition}
\newtheorem{cor}[theorem]{Corollary}

\DeclareMathOperator{\im}{im}

\DeclareMathOperator{\sfl}{sf}

\DeclareMathOperator{\sgn}{sgn}

\DeclareMathOperator{\diverg}{div}

\DeclareMathOperator{\re}{Re}

\title{On bifurcation for semilinear elliptic Dirichlet problems on shrinking domains}
\author{Nils Waterstraat}

\begin{document}
\date{}
\maketitle

\footnotetext[3]{N. Waterstraat was supported by a postdoctoral fellowship of
the German Academic Exchange Service (DAAD) and by the program "Professori Visitatori Junior" of GNAMPA-INdAM.}

\begin{abstract}
We consider the Dirichlet problem for semilinear elliptic equations on a bounded domain which is diffeomorphic to a ball and investigate bifurcation from a given (trivial) branch of solutions, where the radius of the ball serves as bifurcation parameter. Our methods are based on well known results from variational bifurcation theory, which we outline in a separate section for the readers' convenience. 
\end{abstract}

\section{Introduction}
Let $\Omega\subset\mathbb{R}^n$ be a bounded domain and

\[\Phi:\overline{\Omega}\rightarrow B[0,1]\]
a diffeomorphism, where $B[0,1]$ denotes the closed unit ball around $0\in\mathbb{R}^n$. We consider the Dirichlet boundary value problem 

\begin{equation}\label{bvp}
\left\{
\begin{aligned}
Lu(x)+g(x,u(x))&=0,\quad x\in\Omega\\
u(x)&=0,\quad x\in\partial\Omega,
\end{aligned}
\right.
\end{equation}
where

\[Lu(x)=-\sum^n_{i,j=1}{\frac{\partial}{\partial x_j}\left(a_{ij}(x)\frac{\partial u}{\partial x_i}(x)\right)},\quad x\in\Omega,\]
for some smooth functions $a_{ij}:\overline{\Omega}\rightarrow\mathbb{R}$, $a_{ij}=a_{ji}$, $1\leq i,j\leq n$, which satisfy the ellipticity condition

\begin{align}\label{ellipticity}
\sum^n_{i,j=1}{a_{ij}(x)\xi_i\xi_j>0,\quad x\in\Omega,\,\,(\xi_1,\ldots,\xi_n)\in\mathbb{R}^n\setminus\{0\}}.
\end{align}
Moreover, we assume that $g:\overline{\Omega}\times\mathbb{R}\rightarrow\mathbb{R}$ is continuously differentiable, $g(x,0)=0$, $x\in\Omega$, and we suppose that there are constants $\alpha, C$ such that 

\begin{align}\label{growth}
\left|\frac{\partial g}{\partial\xi}(x,\xi)\right|\leq C(1+|\xi|^{\alpha-1}),\quad (x,\xi)\in\Omega\times\mathbb{R}, 
\end{align}
where $1\leq\alpha\leq\frac{n+2}{n-2}$ if $n\geq 3$ and $1\leq\alpha<\infty$ if $n=2$. Finally, in the case $n=1$, that is, \eqref{bvp} is an ordinary differential equation, we do not impose a growth condition on the nonlinearity $g$.\\
Let us denote for $0<r\leq 1$ by $B(0,r)$ the open ball of radius $r$ around $0$ in $\mathbb{R}^n$ and by 

\begin{align}\label{radiusball}
\Omega_r:=\Phi^{-1}(B(0,r))
\end{align}
the domain $\Omega$ shrinked by the factor $r$. We consider the corresponding boundary value problems

\begin{equation}\label{bvpresc}
\left\{
\begin{aligned}
Lu(x)+g(x,u(x))&=0,\quad x\in\Omega_r\\
u(x)&=0,\quad x\in\partial\Omega_r.
\end{aligned}
\right.
\end{equation}
Note that $u\equiv 0$ is a solution of \eqref{bvpresc} for all instants $r$. We call $r_0\in(0,1]$ a \textit{bifurcation instant} for the boundary value problems \eqref{bvpresc} if 
there exists a sequence of radii $r_n\rightarrow r_0$ and functions $u_n\in H^1_0(\Omega_{r_n})$ such that $u_n$ is a non-trivial weak solution of \eqref{bvpresc} on $\Omega_{r_n}$ and $\|u_n\|_{H^1_0(\Omega_{r_n})}\rightarrow 0$.\\ Let us point out that we exclude from the definition the limiting case $r_0=0$ in which the domain $\Omega_r$ degenerates to a point. The reason is that if $r_n\rightarrow 0$, then $\|u_n\|_{H^1_0(\Omega_{r_n})}\rightarrow 0$, for example, for any sequence of functions $\{u_n\}_{n\in\mathbb{N}}\subset H^1_0(\Omega_{r_n})$ such that $\|\frac{\partial u_n}{\partial x_i}\|_{L^\infty(\Omega_r)}<\infty$, $n\in\mathbb{N}$, $i=1,\ldots, n$. Consequently, a bifurcation instant $r_0=0$ would not imply the existence of non-trivial solutions of \eqref{bvpresc} for small $r>0$ which are arbitrarily close to the trivial solution $u\equiv 0$ in a suitable sense.\\
In what follows, we suppose that the function $f:\Omega\rightarrow\mathbb{R}$ defined by $f(x)=\frac{\partial g}{\partial \xi}(x,0)$ is smooth, and we consider the linearised boundary value problems

\begin{equation}\label{bvpII}
\left\{
\begin{aligned}
Lu(x)+f(x)u(x)&=0,\quad x\in\Omega_r\\
u(x)&=0,\quad x\in\partial\Omega_r.
\end{aligned}
\right.
\end{equation}
We call $r_0\in(0,1]$ a \textit{conjugate instant} for \eqref{bvpII} if the dimension of the space of classical solutions

\begin{align}\label{multiplicity}
m(r_0):=\dim\{u\in C^2(\Omega_r)\cap C(\overline{\Omega_r}):\, u\,\,\,\text{solves}\,\,\,\eqref{bvpII}\}
\end{align}
is non-zero, and from now on we assume that $m(1)=0$.  Our main result reads as follows:

\begin{theorem}\label{theorem}
The bifurcation instants of \eqref{bvpresc} are precisely the conjugate instants of \eqref{bvpII}. 
\end{theorem}
Our proof of Theorem \ref{theorem} uses crossing forms from variational bifurcation theory as follows: Let $H$ be a real separable Hilbert space and $\psi:I\times H\rightarrow\mathbb{R}$ a $C^2$-function, where $I=[0,1]$ denotes the unit interval. We assume that $0\in H$ is a critical point of all functionals $\psi_\lambda:=\psi(\lambda,\cdot)$, $\lambda\in I$, and that the Riesz representations $L_\lambda$ of the second derivatives $D^2_0\psi_\lambda:H\times H\rightarrow \mathbb{R}$ at this critical point are Fredholm operators and have finite Morse indices $\mu_-(L_\lambda)$. It is well known that, if $L_a$ and $L_b$ are invertible for some $0\leq a<b\leq 1$ and $\mu_-(L_a)\neq\mu_-(L_b)$, then there exists a bifurcation instant $\lambda_0\in(a,b)$ for critical points of $\psi$, that is, every neighbourhood of $(\lambda_0,0)$ in $I\times H$ contains elements $(\lambda,u)$, where $u\neq 0$ is a critical point of $\psi_\lambda$. In other words, a jump in the Morse index along the path $L=\{L_\lambda\}$ entails bifurcation of critical points of $\psi$. Besides this result, we recall below in Section 2 that, if $L$ is continuously differentiable with respect to the parameter $\lambda\in I$, then jumps in the Morse index can be computed as follows: assume that $L_{\lambda_0}$ has a non-trivial kernel for some $\lambda_0\in (0,1)$. The derivative $\dot L_{\lambda_0}$ of $L$ at $\lambda_0$ plugged into the scalar product on $H$ defines a quadratic form on $H$ and $\lambda_0$ is called regular if its restriction to $\ker L_{\lambda_0}$ is non-degenerate. We recall from \cite{SFLPejsachowicz} that if $\lambda_0$ is a regular crossing of $L$, then there exists $\varepsilon>0$ such that $L_\lambda$ is invertible for all $0<|\lambda-\lambda_0|<\varepsilon$, and the jump in the Morse index when $\lambda$ passes $\lambda_0$ is given by the signature of the quadratic form $\langle\dot L_{\lambda_0}\cdot,\cdot\rangle$ on $\ker L_{\lambda_0}$. This result traces back to the seminal paper \cite{Robbin} of Robbin and Salamon.\\  
The third section is devoted to the proof of Theorem \ref{theorem}. We introduce a family of functionals $\psi:I\times H^1_0(B(0,1))\rightarrow\mathbb{R}$ parametrised by the radius $r$ of the ball in \eqref{radiusball}, such that critical points of $\psi_r:=\psi(r,\cdot)$ correspond to weak solutions of \eqref{bvpresc} under an obvious rescaling mapping from $H^1_0(B(0,1))$ to $H^1_0(\Omega_r)$. Since $u\equiv 0$ is a solution of all equations \eqref{bvpresc}, $0\in H^1_0(B(0,1))$ is a critical point of all functionals $\psi_r$, $r\in I$, and $r_0\in (0,1]$ is a bifurcation instant for $\psi$ if and only if it is a bifurcation instant for the equations \eqref{bvpresc}. Moreover, we explain below that the kernels of the associated operators $L_r$ induced by the second derivative of $\psi_r$ at the critical point $0$ correspond to solutions of the linearised boundary value problems \eqref{bvpII}. The main part of our proof shows that each crossing $r_0\in(0,1)$ of this path $L=\{L_r\}$ is regular and that the signature of the restriction of $\langle\dot L_{r_0}\cdot,\cdot\rangle$ to $\ker L_{r_0}$ is given by the dimension $m(r_0)$ introduced in \eqref{multiplicity}. Consequently, if $m(r_0)\neq 0$, then $r_0$ is a bifurcation instant for the equations \eqref{bvpresc}. Since every bifurcation instant is easily seen to be conjugate by the implicit function theorem, this will prove Theorem \ref{theorem}.\\
In the final Section 4, we show at first a corollary that we obtain from our proof of Theorem \ref{theorem} as outlined in the previous paragraph. Since each crossing $r_0\in(0,1)$ of $L$ is regular, we deduce that there exist only finitely many instants at which $\ker L_r\neq\{0\}$, and hence $m(r)=\dim\ker L_r=0$ for all but finitely many $r\in(0,1)$. Let $\mu_-$ denote the Morse index of the linearised equation \eqref{bvpII} on $\Omega=\Omega_1$, i.e. the number of negative eigenvalues

\begin{equation}\label{bvpIII}
\left\{
\begin{aligned}
Lu(x)+f(x)u(x)&=\lambda\,u(x),\quad x\in\Omega\\
u(x)&=0,\quad x\in\partial\Omega,
\end{aligned}
\right.
\end{equation}
counted according to their multiplicities. We derive from our proof of Theorem \ref{theorem} that

\begin{align}\label{Smale}
\mu_-=\sum_{0<r<1}{m(r)}.
\end{align} 
Let us point out that this result was already obtained by Smale in \cite{Smale} (cf. also \cite{SmaleCorr}) for general strongly elliptic differential operators on vector bundles over compact manifolds with boundary, which is a generalisation of the Morse index theorem for geodesics in Riemannian manifolds to partial differential equations and motivates the definition of conjugate instants for \eqref{bvpII}. However, Smale's argument is based on the domain monotonicity of eigenvalues and consequently it is rather different from our way to equality \eqref{Smale}.\\
We derive from \eqref{Smale} the following corollary of Theorem \ref{theorem}.

\begin{cor}\label{cor}
 If $\mu_-\neq 0$, then there exist at least
	
	\begin{align*}
	\left\lfloor\frac{\mu_-}{\max_{0<r<1}m(r)}\right\rfloor
	\end{align*}
distinct bifurcation instants in $(0,1)$, where $\lfloor\cdot\rfloor$ denotes the integral part of a real number.
\end{cor}
We conclude the fourth section by some examples of our theory including ordinary differential equations and a well known equation from geometric analysis.\\
Finally, let us point out that Theorem \ref{theorem} and Corollary \ref{cor} have appeared in various generality in our joint articles \cite{AleIchDomain} and \cite{AleIchBall} with A. Portaluri from Universit\`a degli studi di Torino. In \cite{AleIchDomain} we assume that $L$ is the Laplacian on a star-shaped domain in $\mathbb{R}^n$, whereas in \cite{AleIchBall} we consider the Laplace-Beltrami operator on a geodesic ball of a Riemannian manifold. Of course, after introducing coordinates in the latter case, both results are special cases of the slightly more general setting that we consider here. A further aim of this article is to give a thorough exposition on crossing forms in bifurcation theory as outlined above, which should make the presentation self-contained in contrast to our previous works \cite{AleIchDomain} and \cite{AleIchBall}, which strongly rely on the article \cite{SFLPejsachowicz} of Fitzpatrick, Pejsachowicz and Recht.


\section{Bifurcation of critical points for essentially positive functionals}
The aim of this section is to explain the necessary basics from variational bifurcation theory that we need in order to prove Theorem \ref{theorem} in the following section. Before we come to bifurcation theory, we recap in a first subsection some facts from spectral theory for selfadjoint Fredholm operators on real Hilbert spaces. Subsequently, we recall the well known principle from nonlinear functional analysis that jumps in the Morse index of the Hessians entail bifurcation of critical points of families of functionals (cf. Theorem \ref{bifclas}). Since this result can be found in common textbooks, we do not prove it here, but just quote some references. The final part of this section is concerned with crossing forms and their relation to jumps of the Morse index along paths of selfadjoint Fredholm operators. The main result is stated in Proposition \ref{prop}, which finally leads to the bifurcation theorem \ref{bifcross} on which the proof of our main theorem \ref{theorem} relies. Since the only reference for Proposition \ref{prop} that we are aware of is the more general Theorem 4.1 in \cite{SFLPejsachowicz}, we provide full details in this section and in particular include a proof of Lemma \ref{lemma-isolated}, which was left to the reader in \cite{SFLPejsachowicz}.


\subsection{Selfadjoint Fredholm operators and their spectra}\label{section-specproj}
Let $H$ be a real Hilbert space. We denote by $\mathcal{L}(H)$ the space of bounded linear operators, by $GL(H)\subset\mathcal{L}(H)$ the open subset of invertible operators, and by $\mathcal{S}(H)\subset\mathcal{L}(H)$ the closed subset of selfadjoint operators. For $T\in\mathcal{S}(H)$, the spectrum 

\[\sigma(T)=\{\lambda\in\mathbb{R}:\,\lambda-T\notin GL(H)\}\]
is non-empty, and our first aim is to introduce spectral projections for $T$. Before we do this, we briefly discuss complexifications of real Hilbert spaces and their operators.\\
The complexification $H^\mathbb{C}$ of $H$ is the linear space consisting of all formal elements $u+iv$, where $u,v\in H$. $H^\mathbb{C}$ is a Hilbert space with respect to the scalar product 

\[\langle u_1+iv_1,u_2+iv_2\rangle_{H^\mathbb{C}}:=\langle u_1,u_2\rangle_H+\langle v_1,v_2\rangle_H+i\langle v_1,u_2\rangle_H-i\langle u_1,v_2\rangle_H.\]  
Each operator $T\in\mathcal{L}(H)$ induces canonically a bounded linear operator on $H^\mathbb{C}$ by $T^\mathbb{C}(u+iv)=Tu+iTv$, and conversely, if $A$ is a bounded linear operator on $H^\mathbb{C}$, it is readily seen that there are unique operators $T,S$ in $\mathcal{L}(H)$ such that $A=T^\mathbb{C}+iS^\mathbb{C}$. As a consequence, there exists a conjugation $\overline{\cdot}$ on the Banach space $\mathcal{L}(H^\mathbb{C})$ of all bounded linear operators on $H^\mathbb{C}$, and $A=T^\mathbb{C}$ for some $T\in\mathcal{L}(H)$ if and only if $\overline A=A$. This particularly holds for the real part $\re(A):=\frac{1}{2}(A+\overline A)$ of an operator $A\in\mathcal{L}(H^\mathbb{C})$.\\
Let now $T\in\mathcal{S}(H)$, and $a,b\notin\sigma(T)$ such that $[a,b]\cap\sigma(T)$ is a non-empty finite set consisting solely of eigenvalues of finite type; i.e., $0<\dim\ker(\lambda-T)<\infty$ for all $\lambda\in[a,b]\cap\sigma(T)$. Let $D\subset\mathbb{C}$ be a disc such that $\partial D\cap\mathbb{R}=\{a,b\}$. It is readily seen that $\sigma(T)=\sigma(T^\mathbb{C})$ and that $(\ker(\lambda-T))^\mathbb{C}=\ker(\lambda-T^\mathbb{C})$. In particular, $\sigma(T^\mathbb{C})\cap D$ consists only of eigenvalues of finite type and $\partial D\cap\sigma(T^\mathbb{C})=\emptyset$. Let us recall that the spectral projection of $T^\mathbb{C}$ with respect to $D\cap\sigma(T^\mathbb{C})$ is the orthogonal projection 

\begin{align}\label{specproj}
P(T^\mathbb{C})=\frac{1}{2\pi i}\int_{\partial D}{(\lambda-T^\mathbb{C})^{-1}\,d\lambda},
\end{align} 
which projects onto the direct sum of all eigenspaces $\ker(\lambda-T^\mathbb{C})$ of $T^\mathbb{C}$ for eigenvalues $\lambda\in D$. Using once again that $\sigma(T)=\sigma(T^\mathbb{C})$ and $\ker(\lambda-T^\mathbb{C})=(\ker(\lambda- T))^\mathbb{C}$, it follows that $P_{[a,b]}(T):=\re(P(T^\mathbb{C}))$ is the orthogonal projection in $H$ onto the direct sum of the eigenspaces of $T$ with respect to the eigenvalues in $[a,b]$. We deduce from the construction of $P_{[a,b]}(T)$ and the corresponding result for linear operators on complex Hilbert spaces (cf. eg. \cite[Thm. II.4.2]{Gohberg}) the following continuity property for isolated eigenvalues of finite multiplicity.
  
\begin{lemma}\label{stabspec}
Let $T\in\mathcal{S}(H)$ and $a,b\in\mathbb{R}$ such that $a,b\notin\sigma(T)$. Assume that $(a,b)\cap\sigma(T)$ consists of isolated eigenvalues of finite multiplicity. Then there exists $\varepsilon>0$ such that $a,b\notin\sigma(S)$ and

\[\dim\im P_{[a,b]}(S)=\dim\im P_{[a,b]}(T)\]
for all $S\in\mathcal{S}(H)$ such that $\|T-S\|_{\mathcal{L}(H)}<\varepsilon$.
\end{lemma}

In what follows, we denote by $\mathcal{FS}(H)$ the set of all elements in $\mathcal{S}(H)$ which are Fredholm, and we note that an operator in $\mathcal{S}(H)$ belongs to $\mathcal{FS}(H)$, if and only if its kernel is of finite dimension and its image is closed.  

\begin{lemma}\label{0isolated}
If $T\in\mathcal{FS}(H)$, then either $0$ belongs to the resolvent set of $T$ or it is an isolated eigenvalue of finite multiplicity.
\end{lemma}

\begin{proof}
Since the set of Fredholm operators is open in $\mathcal{L}(H)$, there exists $\varepsilon>0$ such that $\lambda-T$ is Fredholm for all $|\lambda|<\varepsilon$. Hence, if $|\lambda|<\varepsilon$, then either $\lambda-T$ is invertible or it has a finite dimensional kernel. It remains to show that $0$ is isolated in $\sigma(T)$. Since $T$ is selfadjoint and Fredholm, there is an orthogonal decomposition $H=\ker T\oplus\im T$, and the restriction $T'$ of $T$ to the closed subspace $(\ker T)^\perp=\im T$ is an isomorphism onto $\im T$. Since $GL(\im T)\subset\mathcal{L}(H)$ is open, there exists $\kappa>0$ such that $\sigma(T')\cap(-\kappa,\kappa)=\emptyset$. Let now $0<|\lambda|<\min\{\kappa,\varepsilon\}$ and let $u=u_1+u_2\in\ker T\oplus\im T$ be an element of $\ker(\lambda-T)$. Then $0=\lambda u-Tu=\lambda u_1+\lambda u_2-Tu_2$ and so 

\[\lambda u_1=(T-\lambda)u_2=(T'-\lambda)u_2.\]
Since the left hand side of this equality is in $\ker T$ and the right hand side is in $\im T=(\ker T)^\perp$, we conclude that $\lambda u_1=(T'-\lambda)u_2=0$. From $\lambda\neq 0$ and the invertibility of $T'-\lambda$, it follows that $u=u_1+u_2=0$.    
\end{proof}

As a final piece of notation, we let $\mathcal{FS}_+(H)$ be the set of all elements $T\in\mathcal{FS}(H)$ such that $\sigma(T)\cap(-\infty,0)$ consists of a finite number of eigenvalues of finite type. The Morse index

\[\mu_-(T)=\dim\left(\bigoplus_{\lambda<0}\ker(\lambda-T)\right)\]
is finite for all $T\in\mathcal{FS}_+(H)$ and we note for later reference the following elementary properties:

\begin{enumerate}
	\item[i)] If $U:H\rightarrow H$ is an orthogonal operator, then $UTU^{-1}\in\mathcal{FS}_+(H)$ and $\mu_-(U TU^{-1})=\mu_-(T)$.
	\item[ii)] If $T$ is reduced by a splitting $H=H_0\oplus H_1$, that is, $T(H_i)\subset H_i$, $i=0,1$, then 
	
	\[\mu_-(T)=\mu_-(T\mid_{H_0})+\mu_-(T\mid_{H_1}).\]
	\item[iii)] If $T_0,T_1$ belong to the same component of $\mathcal{FS}_+(H)\cap GL(H)$, then $\mu_-(T_1)=\mu_-(T_0)$. 
\end{enumerate}
Note that the first two properties follow immediately from the definition, whereas the last one is a consequence of Lemma \ref{stabspec}.\\
Finally, let us mention for the sake of completeness that Atiyah and Singer proved in \cite{AtiyahSinger} that   $\mathcal{FS}(H)$ consists of three components 

\[\mathcal{FS}(H)=\mathcal{FS}_+(H)\cup\mathcal{FS}_\ast(H)\cup\mathcal{FS}_-(H),\]
where an operator $L\in\mathcal{FS}(H)$ belongs to $\mathcal{FS}_{\pm}(H)$ if and only if $\mu_-(\pm L)<\infty$. Moreover, the spaces $\mathcal{FS}_{\pm}(H)$ are contractible, whereas $\mathcal{FS}_\ast(H)$ is a classifying space for the $KO$-theory functor $KO^{-7}$.


\subsection{The bifurcation theorem}

\subsubsection{A classical bifurcation theorem}\label{section-bifurcation}
As in the previous section, let $H$ be a real Hilbert space and recall from the introduction that we denote by $I$ the unit interval $[0,1]$. Let $\psi:I\times H\rightarrow\mathbb{R}$ be a continuous function such that the derivatives $D_u\psi_\lambda$ and $D^2_u\psi_\lambda$ of $\psi_\lambda:=\psi(\lambda,\cdot):H\rightarrow\mathbb{R}$ exist and depend continuously on $(\lambda,u)\in I\times H$. In what follows, we assume that $0\in H$ is a critical point of all functionals $\psi_\lambda$, $\lambda\in I$. 

\begin{defi}\label{defi-bifurcationfunctional}
An instant $\lambda_0\in I$ is called a bifurcation point of critical points of $\psi$ if any neighbourhood of $(\lambda_0,0)\in I\times H$ contains elements $(\lambda,u)$ such that $u\neq 0$ is a critical point of $\psi_\lambda$.
\end{defi} 
The bilinear forms $D^2_0\psi_\lambda:H\times H\rightarrow\mathbb{R}$ given by the second derivative at the critical point $0\in H$ define by the Riesz representation theorem a unique path $L:I\rightarrow\mathcal{S}(H)$ such that

\[D^2_0\psi_\lambda(u,v)=\langle L_\lambda u,v\rangle_H,\quad u,v\in H.\]
The following theorem is a standard result in bifurcation theory.

\begin{theorem}\label{bifclas}
Assume that $L_\lambda\in\mathcal{FS}_+(H)$ for all $\lambda\in I$, so that in particular the Morse index $\mu_-(L_\lambda)$ is finite for all $\lambda\in I$. Let $a,b\in I$, $a<b$, such that $L_a,L_b$ are invertible. If $\mu_-(L_a)\neq\mu_-(L_b)$, then there exists a bifurcation point of critical points in $(a,b)$. 
\end{theorem}

Different proofs of this theorem can be found for example in \cite[\S 8.9]{Mawhin}, by using the continuity of the critical groups in Morse theory, and in \cite[\S II.7.1]{Kielhoeffer}, where the argument is based on Conley index theory for flows of vector fields.\\
Note that a corresponding assertion holds for families of functionals $\psi$ such that $L_\lambda\in\mathcal{FS}_-(H)$ and $\mu_-(-L_a)\neq\mu_-(-L_b)$, which clearly follows by applying Theorem \ref{bifclas} to $-\psi_\lambda$, $\lambda\in I$. However, the situation turns out to be much more involved if $L_\lambda\in\mathcal{FS}_\ast(H)$, $\lambda\in I$, since in this case neither $\mu_-(L_\lambda)$ nor $\mu_-(-L_\lambda)$ is finite. Atiyah, Patodi and Singer constructed in \cite{AtiyahPatodi} in connection with spectral asymmetry and the $\eta$-invariant an integer valued homotopy invariant for paths in any component of $\mathcal{FS}(H)$ which is called \textit{spectral flow}. Roughly speaking, the spectral flow $\sfl(L,[a,b])$ of a path $L:[a,b]\rightarrow\mathcal{FS}(H)$ is the number of negative eigenvalues of $L_a$ that become positive as the parameter $\lambda$ travels from $a$ to $b$ minus the number of positive eigenvalues of $L_a$ that become negative. If $L_\lambda\in\mathcal{FS}_+(H)$, $\lambda\in [a,b]$, then $\sfl(L,[a,b])=\mu_-(L_a)-\mu_-(L_b)$ and so the non-vanishing of the spectral flow entails bifurcation in Theorem \ref{bifclas} if $L$ is induced by the second derivative of a family of functionals $\psi$ as above. Accordingly, one may ask if the assertion of Theorem \ref{bifclas} remains to be true if $L_\lambda\in\mathcal{FS}_\ast(H)$, $\lambda\in I$, and $\sfl(L,[a,b])\neq 0$. This question was answered in the affirmative by Fitzpatrick, Pejsachowicz and Recht in \cite{SFLPejsachowicz} under the additional assumption that the entire map $\psi:I\times H\rightarrow\mathbb{R}$ is $C^2$, and it was improved recently in the joint work \cite{JacoboIch} of Pejsachowicz and the author to continuous maps $\psi:I\times H\rightarrow\mathbb{R}$ that satisfy the same differentiability assumptions as in Theorem \ref{bifclas}, that is, each $\psi_\lambda$ is $C^2$ and its derivatives depend continuously on the parameter $\lambda\in I$.


\subsubsection{Morse index and crossing forms}
Let us now consider a continuously differentiable path $L:I\rightarrow\mathcal{FS}_+(H)$ and let us denote by $\dot L_{\lambda_0}$ the derivative of $L$ with respect to $\lambda$ at $\lambda_0$. We follow the notation of Robbin and Salamon in \cite{Robbin} and denote by $\Gamma(L,\lambda_0)$ the quadratic form

\[\Gamma(L,\lambda_0)[u]=\langle \dot L_{\lambda_0}u,u\rangle_H,\quad u\in\ker L_{\lambda_0},\]
on the finite dimensional space $\ker L_{\lambda_0}$, which may be zero-dimensional.

\begin{defi}
An instant $\lambda_0\in I$ is called a crossing of the path $L$ if $\ker L_{\lambda_0}\neq 0$, and a crossing $\lambda_0$ is called regular if $\Gamma(L,\lambda_0)$ is non-degenerate.
\end{defi}  

We begin by showing that regular crossings are isolated, which is an immediate consequence of the following lemma.

\begin{lemma}\label{lemma-isolated}
If $\lambda_0$ is a regular crossing of $L$, then there exist $\varepsilon, C>0$ such that for $|\lambda-\lambda_0|<\varepsilon$

\[\|L_\lambda u\|\geq C\,|\lambda-\lambda_0|\,\|u\|,\quad u\in H.\]
\end{lemma}

\begin{proof}
Since $L_{\lambda_0}$ is a selfadjoint Fredholm operator, there is an orthogonal decomposition $H=\ker L_{\lambda_0}\oplus\im L_{\lambda_0}=:H_0\oplus H_1$. Let us assume at first that $u\in H_1$ and let $P^\perp_{\lambda_0}$ denote the orthogonal projection onto $H_1=H^\perp_0$. Since $P^\perp_{\lambda_0}L_{\lambda_0} P^\perp_{\lambda_0}$ is invertible and $GL(H_1)\subset\mathcal{L}(H_1)$ is open, there exists $\varepsilon'>0$ such that $P^\perp_{\lambda_0}L_\lambda P^\perp_{\lambda_0}\in\mathcal{L}(H_1)$ is invertible for all $|\lambda-\lambda_0|<\varepsilon'$. From $\|P^\perp_{\lambda_0}\|=1$ and $P^\perp_{\lambda_0} u=u$ for all $u\in H_1$, we conclude that there is $C>0$ such that

\[\|L_\lambda u\|\geq \|P^\perp_{\lambda_0} L_\lambda P^\perp_{\lambda_0} u\|\geq C\, \|u\|\geq C\,|\lambda-\lambda_0|\|u\|,\quad u\in H_1,\] 
for all $|\lambda-\lambda_0|<\varepsilon$, where $0<\varepsilon<\min\{1,\varepsilon'\}$.\\
Let us now assume that $u\in H_0$ and let us denote by $P_{\lambda_0}$ the orthogonal projection onto $H_0$. From

\[\Gamma(L,\lambda_0)=\langle P_{\lambda_0}\dot L_{\lambda_0}P_{\lambda_0}u,u\rangle,\quad u\in H_0,\]
and the assumption that $\lambda_0$ is regular, we conclude that $P_{\lambda_0}\dot L_{\lambda_0}P_{\lambda_0}$ is invertible on the finite dimensional space $H_0$. Consequently, there exists $\varepsilon >0$ such that

\[\frac{P_{\lambda_0}L_\lambda P_{\lambda_0}}{\lambda-\lambda_0}=\frac{P_{\lambda_0}L_\lambda P_{\lambda_0}-P_{\lambda_0}L_{\lambda_0} P_{\lambda_0}}{\lambda-\lambda_0}:H_0\rightarrow H_0\]
is invertible for all $|\lambda-\lambda_0|<\varepsilon$. Since $\|P_{\lambda_0}\|=1$ and $P_{\lambda_0}u=u$, $u\in H_0$, there exists $C>0$ such that

\[\|L_\lambda u\|\geq\|P_{\lambda_0}L_\lambda P_{\lambda_0}u\|\geq C\,|\lambda-\lambda_0|\|u\|,\quad u\in H_0.\]
 
\end{proof}

In what follows, we denote by $\sgn\Gamma(L,\lambda_0)$ the signature of the quadratic form $\Gamma(L,\lambda_0)$. Let us recall that $\sgn\Gamma(L,\lambda_0)$ is the difference of the number of positive and negative eigenvalues of the symmetric operator $P_{\lambda_0}\dot L_{\lambda_0}P_{\lambda_0}$ on the finite dimensional space $\ker L_{\lambda_0}$, where $P_{\lambda_0}$ denotes the orthogonal projection onto $\ker L_{\lambda_0}$ as in the previous proof. In other words,

\[\sgn\Gamma(L,\lambda_0)=\mu_-(-P_{\lambda_0}\dot L_{\lambda_0}P_{\lambda_0})-\mu_-(P_{\lambda_0}\dot L_{\lambda_0}P_{\lambda_0}).\]
Let us point out that the following proposition holds for any path $L$, irrespective if it is induced by the Hessians of a family of functionals as in the previous section.

\begin{prop}\label{prop}
If $L_a$, $L_b$ are invertible for some $a,b\in I$, $a<b$, and $L$ has only regular crossings in $(a,b)$, then 

\[\mu_-(L_a)-\mu_-(L_b)=\sum_{\lambda\in (a,b)}{\sgn\Gamma(L,\lambda)}.\]
\end{prop}

\begin{proof}
By Lemma \ref{lemma-isolated} and property iii) of the Morse index, we can clearly assume that there is only a single regular crossing $\lambda_0$ of $L$ in $(a,b)$. Then $\ker L_{\lambda_0}\neq\{0\}$, and we infer from Lemma \ref{0isolated} that there exists $\varepsilon>0$ such that $0$ is the only eigenvalue of $L_{\lambda_0}$ in the interval $[-\varepsilon,\varepsilon]$. Moreover, by Lemma \ref{stabspec}, there exists $\rho>0$ such that $\pm\varepsilon$ is not in the spectrum of $L_\lambda$ for all $\lambda\in[\lambda_0-\rho,\lambda_0+\rho]\subset[a,b]$. For simplicity of notation, we write $P_\lambda:=P_{[-\varepsilon,\varepsilon]}(L_\lambda)$ for the orthogonal projection onto the sum of the eigenspaces of $L_\lambda$ with respect to eigenvalues in $[-\varepsilon,\varepsilon]$ as introduced in Section \ref{section-specproj}. By \eqref{specproj} it is easily seen that that path $P=\{P_\lambda\}:[\lambda_0-\rho,\lambda_0+\rho]\rightarrow\mathcal{L}(H)$ of bounded projections is continuously differentiable, and according to \cite[Sect. VI.2]{Kato} there is an interval $[c,d]\subset[\lambda_0-\rho,\lambda_0+\rho]$ around $\lambda_0$ and a continuously differentiable path $U_\lambda$ of orthogonal operators on $H$ such that

\begin{align}\label{orthogonal}
U_{\lambda_0}=I_H,\quad U_\lambda P_\lambda U^{-1}_\lambda=P_{\lambda_0},\quad\lambda\in[c,d].
\end{align}  
Since $L_\lambda$ commutes with $P_\lambda$, we deduce from \eqref{orthogonal} that $U_\lambda L_\lambda U^{-1}_\lambda$ is reduced by the decomposition $H=H_0\oplus H^\perp_0$, where $H_0=\im P_{\lambda_0}=\ker L_{\lambda_0}$, and moreover, $U_\lambda L_\lambda U^{-1}_\lambda\mid_{H^\perp_0}:H^\perp_0\rightarrow H^\perp_0$ is an isomorphism, $\lambda\in[c,d]$. Consequently, we obtain from the properties i)-iii) of the Morse index 

\begin{align}\label{step1}
\begin{split}
\mu_-(L_a)-\mu_-(L_b)&=\mu_-(L_c)-\mu_-(L_d)=\mu_-(U_cL_cU^{-1}_c)-\mu_-(U_dL_dU^{-1}_d)\\
&=\mu_-(U_cL_cU^{-1}_c\mid_{H_0})+\mu_-(U_cL_cU^{-1}_c\mid_{H^\perp_0})\\
&-(\mu_-(U_dL_dU^{-1}_d\mid_{H_0})+\mu_-(U_dL_dU^{-1}_d\mid_{H^\perp_0}))\\
&=\mu_-(U_cL_cU^{-1}_c\mid_{H_0})-\mu_-(U_dL_dU^{-1}_d\mid_{H_0}).
\end{split}
\end{align}
Let us now consider $\ell_\lambda:=U_\lambda L_\lambda U^{-1}_\lambda\mid_{H_0}$, $\lambda\in[c,d]$, which is a continuously differentiable path of symmetric operators on the finite dimensional space $H_0$. Clearly,

\begin{align*}
\dot\ell_\lambda=\frac{d}{d\lambda}(U_\lambda)\circ L_\lambda\circ U^{-1}_\lambda\mid_{H_0}+U_\lambda\circ\frac{d}{d\lambda}(L_\lambda)\circ U^{-1}_\lambda\mid_{H_0}+U_\lambda\circ L_\lambda\circ\frac{d}{d\lambda}(U^{-1}_\lambda)\mid_{H_0},\, \lambda\in[c,d],
\end{align*}
and from $U_{\lambda_0}=I_H$, we see that

\begin{align}\label{equcross}
\begin{split}
\langle\dot\ell_{\lambda_0}u,u\rangle&=\langle\dot L_{\lambda_0}u,u\rangle+\langle  L_{\lambda_0}\dot U^{-1}_{\lambda_0}u,u\rangle=\langle\dot L_{\lambda_0}u,u\rangle+\langle\dot U^{-1}_{\lambda_0}u,L_{\lambda_0}u\rangle\\
&=\langle\dot L_{\lambda_0}u,u\rangle=\Gamma(L,\lambda_0)[u],\quad u\in H_0=\ker L_{\lambda_0}.
\end{split}
\end{align}
Consequently, by \eqref{step1} the proposition is shown once we have proven that 

\[\mu_-(\ell_c)-\mu_-(\ell_d)=\sgn\langle\dot\ell_{\lambda_0}\cdot,\cdot\rangle.\]
From \eqref{equcross} and the assumption that $\lambda_0$ is regular, we conclude that the operator $\dot\ell_{\lambda_0}$ is invertible on the finite dimensional space $H_0$. Thus, there is a constant $\alpha>0$ such that $\dot\ell_{\lambda_0}+B$ is invertible for all linear operators $B:H_0\rightarrow H_0$ of norm less than $\alpha$. We define for $\lambda=c,d$ two straight paths by

\[T_\lambda:[0,1]\rightarrow\mathcal{L}(H_0),\quad T_\lambda(t)=t\ell_\lambda+(1-t)(\lambda-\lambda_0)\,\dot\ell_{\lambda_0}.\]
From $\ell_{\lambda_0}=L_{\lambda_0}\mid_{H_0}=0$, we see that

\[T_\lambda(t)=(\lambda-\lambda_0)\left(t\left(\frac{\ell_\lambda-\ell_{\lambda_0}}{\lambda-\lambda_0}-\dot\ell_{\lambda_0}\right)+\dot\ell_{\lambda_0}\right),\, t\in[0,1].\]
By choosing $c<\lambda_0$ and $d>\lambda_0$ in \eqref{orthogonal} closer to $\lambda_0$, we can assume that the norm of $\frac{\ell_\lambda-\ell_{\lambda_0}}{\lambda-\lambda_0}-\dot\ell_{\lambda_0}$ is less than the constant $\alpha$ for $\lambda=c,d$, and consequently, $T_\lambda(t)$ is invertible for $t\in[0,1]$ and $\lambda=c,d$. We obtain

\[\mu_-(\ell_c)=\mu_-((c-\lambda_0)\,\dot\ell_{\lambda_0})=\mu_-(-\dot\ell_{\lambda_0})\quad\text{and}\quad \mu_-(\ell_d)=\mu_-((d-\lambda_0)\,\dot\ell_{\lambda_0})=\mu_-(\dot\ell_{\lambda_0}),\]
and finally 
\[\sgn\langle\dot\ell_{\lambda_0}\cdot,\cdot\rangle=\mu_-(-\dot\ell_{\lambda_0})-\mu_-(\dot\ell_{\lambda_0})=\mu_-(\ell_c)-\mu_-(\ell_d).\]
\end{proof}

Finally, we come back to bifurcation theory and consider a family of functionals $\psi:I\times H\rightarrow\mathbb{R}$ such that $0\in H$ is a critical point of all $\psi_\lambda$, $\lambda\in I$, as in Section \ref{section-bifurcation}. Moreover, we denote as before by $L_\lambda$ the Riesz representation of $D^2_0\psi_\lambda$, and we assume that $L_\lambda\in\mathcal{FS}_+(H)$, $\lambda\in I$. We now obtain from Theorem \ref{bifclas} and Proposition \ref{prop} the following bifurcation result, on which our proof of Theorem \ref{theorem} is based. 

\begin{theorem}\label{bifcross}
If $\lambda_0\in(0,1)$ is a regular crossing of $L$ and $\sgn\Gamma(L,\lambda_0)\neq 0$, then $\lambda_0$ is a bifurcation point of critical points for $\psi$.
\end{theorem}


\section{Elliptic Dirichlet problems on shrinking domains}
The aim of this section is to prove Theorem \ref{theorem}. Let us denote by $\nu=(\nu_1,\ldots,\nu_n)$ the outer normal unit vector field on $\partial\Omega$ and recall that

\begin{align}\label{IBP}
\int_{\Omega}{\frac{\partial u}{\partial x_i}(x)\,v(x)dx}=-\int_{\Omega}{u(x)\frac{\partial v}{\partial x_i}(x)dx}+\int_{\partial\Omega}{u(x)\,v(x)\nu_i(x)\,dS}
\end{align} 
for all $u,v\in C^1(\Omega)\cap C(\overline{\Omega})$. From this integration by parts formula, we clearly see that a weak solution of the boundary value problem \eqref{bvpresc} is a function $\hat u\in H^1_0(\Omega_r)$ such that

\begin{align}\label{weaksol}
\int_{\Omega_r}{\sum^n_{i,j=1}{a_{ij}(x)\frac{\partial\hat u}{\partial x_i}(x)\frac{\partial\hat v}{\partial x_j}(x)}dx}+\int_{\Omega_r}{g( x,\hat u(x))\,\hat v(x)\,dx}=0,\quad \hat v\in H^1_0(\Omega_r).
\end{align}
In what follows, we denote for the sake of simplicity by $B:=B(0,1)$ the unit ball around $0$ in $\mathbb{R}^n$. It is readily seen that, after a change of coordinates, \eqref{weaksol} is equivalent to

\begin{align}\label{weaksolII}
\int_{B}{\sum^n_{k,l=1}{\tilde{a}_{kl}(r\cdot x)\frac{\partial u}{\partial x_k}(x)\frac{\partial v}{\partial x_l}(x)}dx}+r^2\int_{B}{\tilde g(r\cdot x,u(x))\,v(x)\,dx}=0,\quad v\in H^1_0(B),
\end{align}
where 

\begin{align}\label{usol}
u(x)=\hat u(\Phi^{-1}(r\cdot x)),\quad x\in B,
\end{align}
$\tilde{g}$ is continuously differentiable on $\overline{B}\times\mathbb{R}$, and $\tilde{a}_{kl}\in C^\infty(\overline{B})$ are such that $\tilde{a}_{kl}=\tilde{a}_{lk}$, $1\leq k,l\leq n$, and

\begin{align}\label{positive}
\sum^n_{k,l=1}{\tilde{a}_{kl}(x)\xi_k\xi_l>0},\quad x\in B,\,(\xi_1,\ldots,\xi_n)\in\mathbb{R}^n\setminus\{0\}.
\end{align} 
Let us consider the family of functionals $\psi:I\times H^1_0(B)\rightarrow\mathbb{R}$ defined by

\[\psi(r,u)=\frac{1}{2}\int_{B}{\sum^n_{k,l=1}{\tilde a_{kl}(r\cdot x)\,\frac{\partial u}{\partial x_k}(x)\frac{\partial u}{\partial x_l}(x)}\,dx}+r^2\int_{B}{G(r\cdot x,u(x))\,dx},\] 
where

\[G(x,t)=\int^t_0{\tilde g(x,\xi)\,d\xi},\quad(x,t)\in B\times\mathbb{R},\]
is a primitive of $g$. According to \cite[Prop. B.10]{Rabinowitz}, each $\psi_r$ is $C^2$ and the derivative of $\psi_r$ at $u\in H^1_0(B)$ is given by \eqref{weaksolII}, i.e.,

\[(D_u\psi_r)v=\int_{B}{\sum^n_{k,l=1}{\tilde{a}_{kl}(r\cdot x)\frac{\partial u}{\partial x_k}(x)\frac{\partial v}{\partial x_l}(x)}dx}+r^2\int_{B}{\tilde g(r\cdot x,u(x))\,v(x)\,dx},\quad v\in H^1_0(B).\]
Hence $D_u\psi_r=0\in(H^1_0(B))^\ast$ if and only if the function $\hat u$ defined by \eqref{usol} is a weak solution of the boundary value problem \eqref{bvpresc} on $\Omega_r$. In particular, $0\in H^1_0(B)$ is a critical point of all $\psi_r$, and $r^\ast\in(0,1]$ is a bifurcation instant for $\psi$ if and only if it is a bifurcation instant for the boundary value problems \eqref{bvpresc}. Consequently, the study of bifurcation for \eqref{bvpresc} as defined in the introduction boils down to an investigation of bifurcation for $\psi$ in the sense of Definition \ref{defi-bifurcationfunctional}.\\
By \cite[Prop. B.34]{Rabinowitz}, the second derivative of $\psi_r$ at $0\in H^1_0(B)$ is given by

\[D^2_0\psi_r(u,v)=\int_{B}{\sum^n_{k,l=1}{\tilde{a}_{kl}(r\cdot x)\frac{\partial u}{\partial x_k}(x)\frac{\partial v}{\partial x_l}(x)}\,dx}+r^2\int_{B}{\tilde f(r\cdot x)\,u(x)\,v(x)\,dx},\quad u,v\in H^1_0(B),\]
where $\tilde{f}(x)=\frac{\partial\tilde g}{\partial \xi}(x,0)$, $x\in B$. We now define in accordance with the notation in the previous section a bounded symmetric operator $L_r$ by

\[\langle L_r u,v\rangle_{H^1_0(B)}=D^2_0\psi_r(u,v),\quad u,v\in H^1_0(B),\]
which depends continuously differentiable on the parameter $r$. Clearly, $u\in\ker L_r$ if and only if the rescaled function $\hat u$ defined by \eqref{usol} satisfies

\[\int_{\Omega_r}{\sum^n_{i,j=1}{a_{ij}(x)\frac{\partial\hat u}{\partial x_i}(x)\frac{\partial \hat v}{\partial x_j}(x)}}+\int_{\Omega_r}{f(x)\,\hat u(x)\,\hat v(x)\,dx}=0,\quad \hat v\in H^1_0(\Omega_r),\]
which means that $\hat u$ is a weak solution of the linearised boundary value problem \eqref{bvpII}. Since weak solutions of \eqref{bvpII} are smooth by classical regularity theory (cf. \cite[Cor. 8.11]{Gilbarg}), we conclude in particular that

\begin{align}\label{ker=m}
\dim\ker L_r=m(r),\quad r\in(0,1],
\end{align}
where $m(r)$ is the number introduced in \eqref{multiplicity}.\\
In summary, Theorem \ref{theorem} is proven once we have shown that the bifurcation instants $r$ of $\psi$ in $(0,1]$ are precisely the instants for which $\ker L_r\neq 0$.

\begin{lemma}\label{Lfred}
The operators $L_r$, $r\in[0,1]$, are Fredholm. 
\end{lemma}

\begin{proof}
We set $L_r=L_{r,1}+L_{r,2}\in\mathcal{L}(H^1_0(B))$, where

\begin{align*}
\langle L_{r,1}u,v\rangle_{H^1_0(B)}&=\int_{B}{\sum^n_{k,l=1}{\tilde a_{kl}(r\cdot x)\frac{\partial u}{\partial x_k}(x)\frac{\partial v}{\partial x_l}(x)}\,dx},\quad u,v\in H^1_0(B),\\
\langle L_{r,2}u,v\rangle_{H^1_0(B)}&=r^2\int_{B}{\tilde f(r\cdot x)\,u(x)\,v(x)\,dx,\quad u,v\in H^1_0(B)}.
\end{align*}
From the ellipticity condition \eqref{positive} and the boundedness of the functions $a_{kl}$ on $B$, it clearly follows that $L_{r,1}$ is invertible. $L_{r,2}$ extends to a bounded operator on $L^2(B)$ and hence its restriction to $H^1_0(B)$ is compact, because of the compactness of the embedding $H^1_0(B)\hookrightarrow L^2(B)$. Consequently, $L_r$ is a compact perturbation of an invertible operator and hence Fredholm.  
\end{proof}

We now split the proof of Theorem \ref{theorem} into two parts.


\subsubsection*{Every bifurcation instant $r_0$ of critical points for $\psi$ is a crossing of $L$}

The assertion follows easily from the implicit function theorem in Banach spaces (cf. \cite[\S 2.2]{Ambrosetti}). Here we only use a special case, which reads as follows.

\begin{theorem}
Let $X,Y$ be Banach spaces and $F:I\times X\rightarrow Y$ continuous. Assume that the equation

\begin{align}\label{equ}
F(\lambda,x)=0
\end{align}
has a solution $(\lambda_0,x_0)\in (0,1)\times X$, and that the derivative $D_xF_\lambda$ of $F_\lambda:=F(\lambda,\cdot)$ with respect to $x\in X$ exists and depends continuously on $(\lambda,x)\in I\times X$. If $D_{x_0}F_{\lambda_0}\in GL(X,Y)$, then there exists a neighbourhood $U\times V\subset I\times X$ of $(\lambda_0,x_0)$ and a continuous map $f:U\rightarrow V$ such that $f(\lambda_0)=x_0$, $F(\lambda,f(\lambda))=0$ for all $\lambda\in U$ and every solution of \eqref{equ} in $U\times V$ is of the form $(\lambda,f(\lambda))$.  
\end{theorem}

Indeed, let us assume on the contrary that $\ker L_{r_0}=\{0\}$. We define a map

\[F:I\times H^1_0(B)\rightarrow (H^1_0(B))^\ast,\quad F(r,u)=D_u\psi_r\] 
and note that by assumption $F(r,0)=0$ for all $r\in I$. Since $(D_0F_{r_0}u)v=\langle L_{r_0}u,v\rangle$, $u,v\in H^1_0(B)$, and $\ker L_{r_0}=0$ by assumption, we obtain from Lemma \ref{Lfred} that $D_0F_{r_0}\in GL(H^1_0(B),(H^1_0(B))^\ast)$. Consequently, all solutions of the equation $F(r,u)=0$ in a neighbourhood of $(r_0,0)\in I\times H^1_0(B)$ are of the form $(r,0)$.\\
Finally, let us mention that, since $m(1)=0$ by assumption, bifurcation can only occur in the interior of $I$.


\subsubsection*{Every crossing $r_0$ of $L$ is a bifurcation instant of critical points of $\psi$}
Let $r_0$ be a crossing of $L$ and note that $r_0\in(0,1)$ by \eqref{ker=m} and the assumption that $m(1)=0$. Our aim is to show that $r_0$ is regular and that the signature of the corresponding crossing form does not vanish.\\
Let $u\in\ker L_{r_0}$ and let us note for later reference that

\begin{align}\label{crossingform}
\Gamma(L,r_0)[u]=\int_B{\sum^n_{k,l=1}{\langle\nabla\tilde a_{kl}(r_0\cdot x),x\rangle\frac{\partial u}{\partial x_k}\frac{\partial u}{\partial x_l}\,dx}}+\int_B{\frac{d}{d r}\mid_{r=r_0}(r^2\tilde{f}(r\cdot x))u(x)^2\,dx}.
\end{align} 
As already observed above, from $u\in\ker L_{r_0}$, we obtain that $u$ is smooth and so

\begin{align}\label{classol}
\begin{split}
0&=-\sum^n_{k,l=1}{\frac{\partial}{\partial x_k}\left(\tilde a_{kl}(r_0\cdot x)\frac{\partial u}{\partial x_l}(x)\right)}+r^2_0\tilde{f}(r_0\cdot x)u(x)\\
&=-\sum^n_{k,l=1}{r_0\frac{\partial\tilde a_{kl}}{\partial x_k}(r_0\cdot x)\frac{\partial u}{\partial x_l}(x)+\tilde a_{kl}(r_0\cdot x)\frac{\partial^2 u}{\partial x_k\partial x_l}(x)}+r^2_0f(r_0\cdot x)u(x),\quad x\in B.
\end{split}
\end{align}
We now introduce a new function by $v_r(x):=u(\frac{r}{r_0}\cdot x)$, $r\in(0,r_0]$, $x\in B$, and write

\begin{align}\label{udot}
\dot{u}(x):=\frac{d}{dr}\mid_{r=r_0}v_r(x)=\frac{1}{r_0}\langle\nabla u(x),x\rangle,\quad x\in B.
\end{align}
From \eqref{classol}, we infer that

\begin{align*}
&-\sum^n_{k,l=1}{\frac{\partial}{\partial x_k}\left(\tilde{a}_{kl}(r\cdot x)\frac{\partial v_r}{\partial x_l}(x)\right)}+r^2f(r\cdot x)v_r(x)\\
&=\frac{r^2}{r^2_0}\left(-\sum^n_{k,l=1}{(r_0\frac{\partial\tilde a_{kl}}{\partial x_k}(r_0\frac{r}{r_0}\,x)\frac{\partial u}{\partial x_l}(\frac{r}{r_0}\,x)+\tilde a_{kl}(r_0\frac{r}{r_0}\,x)\frac{\partial ^2 u}{\partial x_k\partial x_l}(\frac{r}{r_0}\,x))}+r^2_0\tilde f(r_0\frac{r}{r_0}\,x)u(\frac{r}{r_0}\,x)\right)
\end{align*}
vanishes, and by differentiating with respect to $r$ at $r=r_0$ we get

\begin{align}\label{equ1}
\begin{split}
0&=-\sum^n_{k,l=1}{\frac{\partial}{\partial x_k}\left(\langle\nabla \tilde a_{kl}(r_0\cdot x),x\rangle\frac{\partial u}{\partial x_l}\right)}-\sum^n_{k,l=1}{\frac{\partial}{\partial x_k}\left(\tilde a_{kl}(r_0\cdot x)\frac{\partial\dot{u}}{\partial x_l}\right)}\\
&+\frac{d}{d r}\mid_{r=r_0}(r^2\tilde{f}(r\cdot x))u(x)+r^2_0\tilde{f}(r_0\cdot x)\dot{u}(x),\quad x\in B.
\end{split}
\end{align} 
Now we multiply \eqref{equ1} by $u$ and integrate over $B$:

\begin{align*}
0&=-\int_B{\sum^n_{k,l=1}{\frac{\partial}{\partial x_k}\left(\langle\nabla \tilde a_{kl}(r_0\cdot x),x\rangle\frac{\partial u}{\partial x_l}\right)}u(x)\,dx}-\int_B{\sum^n_{k,l=1}{\frac{\partial}{\partial x_k}\left(\tilde a_{kl}(r_0\cdot x)\frac{\partial\dot{u}}{\partial x_l}\right)}u(x)\,dx}\\
&+\int_B{\frac{d}{dr}\mid_{r=r_0}(r^2\tilde{f}(r\cdot x))u(x)^2\,dx}+\int_B{r^2_0\tilde{f}(r_0\cdot x)\dot{u}(x)u(x)\,dx}.
\end{align*} 
Since $u$ vanishes on $\partial B$, we obtain from \eqref{IBP}

\begin{align*}
0&=\int_B{\sum^n_{k,l=1}{\langle\nabla\tilde a_{kl}(r_0\cdot x),x\rangle\frac{\partial u}{\partial x_l}}\frac{\partial u}{\partial x_k}\,dx}-\int_{\partial B}{\left(\sum^n_{k,l=1}{\langle\nabla \tilde a_{kl}(r_0\cdot x),x\rangle \nu_k(x)\frac{\partial u}{\partial x_l}}\right)u(x)\,dS}\\
&+\int_B{\sum^n_{k,l=1}{\tilde a_{kl}(r_0\cdot x)\frac{\partial\dot{u}}{\partial x_l}}\frac{\partial u}{\partial x_k}\,dx}-\int_{\partial B}{\left(\sum^n_{k,l=1}{\tilde a_{kl}(r_0\cdot x)\nu_k(x)\frac{\partial\dot{u}}{\partial x_l}}\right)u(x)\,dS}\\
&+\int_B{\frac{d}{dr}\mid_{r=r_0}(r^2\tilde{f}(r\cdot x))u(x)^2\,dx}+\int_B{r^2_0\tilde{f}(r_0\cdot x)\dot{u}(x)u(x)\,dx}\\
&=\int_B{\sum^n_{k,l=1}{\langle\nabla\tilde a_{kl}(r_0\cdot x),x\rangle\frac{\partial u}{\partial x_l}}\frac{\partial u}{\partial x_k}\,dx}-\int_B{\sum^n_{k,l=1}{\frac{\partial}{\partial x_l}\left(\tilde a_{kl}(r_0\cdot x)\frac{\partial u}{\partial x_k}\right)}\dot{u}(x)\,dx}\\
&+\int_{\partial B}{\left(\sum^n_{k,l=1}{\tilde a_{kl}(r_0\cdot x)\nu_l(x)\frac{\partial u}{\partial x_k}}\right)\dot{u}(x)\,dS}\\
&+\int_B{\frac{d}{dr}\mid_{r=r_0}(r^2\tilde{f}(r\cdot x))u(x)^2\,dx}+\int_B{r^2_0\tilde{f}(r_0\cdot x)\dot{u}(x)u(x)\,dx}.
\end{align*}
Now we use the first equality in \eqref{classol}, as well as \eqref{crossingform} and \eqref{udot} to conclude that

\begin{align*}
\Gamma(L,r_0)[u]&=-\frac{1}{r_0}\int_{\partial B}{\left(\sum^n_{k,l=1}{\tilde a_{kl}(r_0\cdot x)\nu_l(x)\frac{\partial u}{\partial x_k}}\right)\langle \nabla u(x),x\rangle\,dS},
\end{align*}
which can be written as

\begin{align*}
\Gamma(L,r_0)[u]=-\frac{1}{r_0}\int_{\partial B}{\langle A(r_0\cdot x)x,\nabla u(x) \rangle\, \langle \nabla u(x),x\rangle\,dS},
\end{align*}
where $A(x):=\{\tilde a_{kl}(x)\}$, $x\in B$, and we use that $\nu(x)=x$ for all $x\in\partial B$. Denoting by $(A(r_0\cdot x)x)^T$, $x\in\partial B$, the component of the vector $A(r_0\cdot x)x$ tangential to $\partial B$, we have

\[\langle A(r_0\cdot x)x,\nabla u(x)\rangle=\langle\nabla u(x),x\rangle\,\langle A(r_0\cdot x)x,x\rangle+\langle\nabla u(x),(A(r_0\cdot x)x)^T\rangle\] 
and it follows that 

\begin{align*}
\Gamma(L,r_0)[u]&=-\frac{1}{r_0}\int_{\partial B}{\langle\nabla u(x),x\rangle^2\,\langle A(r_0\cdot x)x,x \rangle\,dS}\\
&-\frac{1}{r_0}\int_{\partial B}{\langle\nabla u(x),x\rangle\,\langle\nabla u(x), (A(r_0\cdot x)x)^T \rangle\,dS}.
\end{align*}
From

\begin{align*}
&\diverg(u(x)\langle x,\nabla u(x)\rangle(A(r_0\,x)x)^T)=\langle x,\nabla u(x)\rangle\langle\nabla u(x),(A(r_0\,x)x)^T\rangle\\
&+u(x)\langle\nabla\langle x,\nabla u(x)\rangle,(A(r_0\,x)x)^T\rangle+u(x)\langle x,\nabla u(x)\rangle\diverg(A(\lambda_0x)x)^T,
\end{align*}
and $u\mid_{\partial\Omega}=0$, we see that 

\[\langle\nabla u(x),x\rangle\,\langle\nabla u(x),(A(r_0\cdot x)x)^T \rangle=\diverg(u(x)\langle x,\nabla u(x)\rangle(A(r_0\cdot x)x)^T),\quad x\in\partial B,\]
and now Stokes' theorem gives

\begin{align}\label{crossfinal}
\Gamma(L,r_0)[u]=-\frac{1}{r_0}\int_{\partial B}{\langle\nabla u(x),x\rangle^2\,\langle A(r_0\cdot x)x,x \rangle\,dS}\leq 0.
\end{align}
Here we use that $A(x)$ is positive definite for all $x\in B$, which follows from the ellipticity condition \eqref{positive}. Finally, even the strict inequality holds in \eqref{crossfinal} if $u\neq 0$. For otherwise, 

\[\langle\nabla u(x),x\rangle=\langle\nabla u(x),\nu(x)\rangle=\frac{\partial u}{\partial\nu}(x)=0,\quad x\in\partial B,\]
and since $u\mid_{\partial B}=0$, this implies $u\equiv 0$ by the uniqueness of the Cauchy problem for elliptic second order operators (cf. \cite{Calderon}, \cite{Hormander}).\\
In summary, we have shown that $\Gamma(L,r_0)$ is negative definite, and so in particular non-degenerate with the non-vanishing signature

\begin{align}\label{crossingformIII}
\sgn\Gamma(L,r_0)=-\dim\ker L_{r_0}=-m(r_0),
\end{align}
where the second equality was already shown in \eqref{ker=m}.
By Theorem \ref{bifcross}, $r_0$ is a bifurcation instant, and so Theorem \ref{theorem} is proven.


\section{Corollaries and examples}
The first aim of this final section is to prove Corollary \ref{cor}, which is a rather immediate consequence of the equality \eqref{Smale}. In order to derive \eqref{Smale} from the proof of Theorem \ref{theorem}, we note at first that the Morse index $\mu_-$ of \eqref{bvpIII} is given by the Morse index $\mu_-(L_1)$ of the operator $L_1\in\mathcal{FS}_+(H^1_0(B))$. Moreover, since $L_0$ is positive by \eqref{positive}, we see that $\mu_-(L_0)=0$. We have obtained in \eqref{crossingformIII} that all crossings $r_0$ of $L$ are regular and $\sgn\Gamma(L,r_0)=-m(r_0)$. Consequently, Lemma \ref{lemma-isolated} and Proposition \ref{prop} show that $m(r)=\dim\ker L_r=0$ for all but finitely many $r\in(0,1]$ and

\[\mu_-=\mu_-(L_1)-\mu_-(L_0)=-\sum_{0<r<1}{\sgn\Gamma(L,r)}=\sum_{0<r<1}{m(r)}\]
which is \eqref{Smale}.\\
Let us point out that \eqref{Smale} was proven by Smale in \cite{Smale} by studying the monotonicity of eigenvalues under shrinking of domains. Hence we have obtained a new proof of Smale's theorem for the boundary value problems \eqref{bvpII}, and moreover, Corollary \ref{cor} is now an immediate consequence of \eqref{Smale} and Theorem \ref{theorem}.\\
Let us now conclude this final section by some examples and remarks. For applying Corollary \ref{cor}, it is necessary to have an upper bound on the kernel dimensions $m(r)$, $r\in(0,1]$, which may be difficult to determine in general. A particular example is given by semilinear ordinary differential equations, where \eqref{bvpresc} and \eqref{bvpII} reduce to

\begin{equation}\label{ODEI}
\left\{
\begin{aligned}
-(a(x)u'(x))'+g(x,u(x))&=0,\quad x\in[0,r]\\
u(0)=u(r)&=0,
\end{aligned}
\right.
\end{equation}
and

\begin{equation}\label{ODEII}
\left\{
\begin{aligned}
-(a(x)u'(x))'+f(x)u(x)&=0,\quad x\in[0,r]\\
u(0)=u(r)&=0.
\end{aligned}
\right.
\end{equation}  
Here $a:[0,1]\rightarrow\mathbb{R}$ is positive and smooth, $g:[0,1]\times\mathbb{R}\rightarrow\mathbb{R}$ is any $C^1$ function such that $g(x,0)=0$, $x\in[0,1]$, and $f(x)=\frac{\partial g}{\partial\xi}(x,0)$ is assumed to be smooth on $[0,1]$. Since clearly $0\leq m(r)\leq 1$, we not only conclude that every conjugate instant of \eqref{ODEII} is a bifurcation instant of \eqref{ODEI}, but we also obtain from Corollary \ref{cor} that there are exactly $\mu_-$ distinct bifurcation instants, where $\mu_-$ is the number of negative eigenvalues of \eqref{ODEII} for $r=1$, i.e. the number of $\lambda<0$ such that 

\begin{equation*}
\left\{
\begin{aligned}
-(a(x)u'(x))'+f(x)u(x)&=\lambda u(x),\quad x\in[0,1]\\
u(0)=u(1)&=0
\end{aligned}
\right.
\end{equation*}
has a non-trivial solution.\\
Let us mention in passing that the computation of the crossing forms of $L$ in the proof of Theorem \ref{theorem} not only simplifies in the case of ordinary differential equations, but also yields a new proof of the Morse index theorem for geodesics in Riemannian manifolds, which can be found in the recent joint work \cite{MorseYet} of the author with A. Portaluri.\\
In higher dimensions, many semilinear equations come from problems in geometric analysis. Let us refer to \cite{Aubin}, \cite{Besse} and just mention as an example on compact Riemannian manifolds $(M,g)$ of dimension $n\geq 3$ the equation

\begin{align}\label{equationcurvature}
4\,\frac{n-1}{n-2}\,\Delta u(p)+s(p)u(p)=\mu\,|u(p)|^\frac{n+2}{n-2},\quad p\in M,
\end{align}
where $\Delta$ denotes the Laplace-Beltrami operator, $s:M\rightarrow\mathbb{R}$ the scalar curvature function and $\mu$ the Yamabe invariant of the metric $g$ on $M$. The significance of \eqref{equationcurvature} is that if $u\in C^\infty(M)$ is a positive solution, then $\tilde{g}=u^\frac{4}{n-2}g$ is a metric of constant scalar curvature on $M$.\\
Let us now assume that $\Phi:\mathcal{U}\rightarrow\mathbb{R}^n$ is a chart of $M$ and consider $\Omega:=\Phi^{-1}(B(0,1))$. In this case the boundary value problems \eqref{bvpresc} are

\begin{equation}\label{scalcurv}
\left\{
\begin{aligned}
4\,\frac{n-1}{n-2}\,\Delta u(p)+s(p)u(p)&=\mu\,|u(p)|^\frac{n+2}{n-2},\quad p\in \Omega_r\\
u(p)&=0,\quad p\in\partial\Omega_r
\end{aligned}
\right.
\end{equation}
and we obtain as corresponding linear equations \eqref{bvpII}

\begin{equation}\label{scalcurvlin}
\left\{
\begin{aligned}
4\,\frac{n-1}{n-2}\,\Delta u(p)+s(p)u(p)&=0,\quad p\in \Omega_r\\
u(p)&=0,\quad p\in\partial \Omega_r.
\end{aligned}
\right.
\end{equation}
We see that the bifurcation instants do not depend on the Yamabe invariant $\mu$. Moreover, if $g$ is already of constant scalar curvature, then the bifurcation instants are entirely determined by the spectrum of the Laplace-Beltrami operator. In particular, there is no bifurcation for \eqref{scalcurv} on manifolds of constant negative scalar curvature.

\newpage

\vspace{1cm}
Nils Waterstraat\\
Institut f\"ur Mathematik\\
Humboldt-Universit\"at zu Berlin\\
Unter den Linden 6\\
10099 Berlin\\
Germany\\
E-mail: waterstn@math.hu-berlin.de

\end{document}